\documentclass[12pt,reqno]{amsart}
\usepackage{amsmath, amssymb, amsthm}
\usepackage{url}
\usepackage[breaklinks]{hyperref}

\usepackage[T1]{fontenc}

\renewcommand{\Re}{\operatorname{Re}}

 \renewcommand{\a}{\alpha}
\renewcommand{\b}{\beta}

\renewcommand{\(}{\left\(}
\renewcommand{\)}{\right\)}
\renewcommand{\[}{\left\[}
\renewcommand{\]}{\right\]}

\newcommand{\al}{\alpha}
\newcommand{\be}{\beta}
\newcommand{\rz}[1]{\zeta(#1)}

\newcommand{\bracket}[1]{\left( #1 \right)}

\newcommand{\Z}{\mathbb{Z}}

\newcommand{\C}{\mathbb{C}}

\newcommand{\summ}[2]{\sum_{#1}^{#2}}
\newcommand{\I}[1]{\int_{#1 -i\infty}^{#1 +i\infty}}
\newcommand{\ls}[1]{L(#1,\chi_1)}

\numberwithin{equation}{section}
 \theoremstyle{plain}
\newtheorem{theorem}{Theorem}[section]
\newtheorem{lemma}[theorem]{Lemma}
\newtheorem{remark}[]{Remark}

\newtheorem{conjecture}[theorem]{Conjecture}

\newtheorem{corollary}[theorem]{Corollary}
\newtheorem{proposition}[theorem]{Proposition}

   \makeatletter
\def\proof{\@ifnextchar[{\@oproof}{\@nproof}}
\def\@oproof[#1][#2]{\trivlist\item[\hskip\labelsep\textit{#2 Proof of\
#1.}~]\ignorespaces}
\def\@nproof{\trivlist\item[\hskip\labelsep\textit{Proof.}~]\ignorespaces}

\makeatother

\begin{document}
\title[A new Ramanujan-type identity for $L(2k+1,  \chi_1)$]{A new Ramanujan-type identity for $L(2k+1, \chi_1)$} 

\author{Shashi Chourasiya}
\address{Shashi Chourasiya \\ Department of Mathematics \\
Indian Institute of Technology Indore \\
Simrol,  Indore,  Madhya Pradesh 453552, India.} 
\email{shashich12@gmail.com}

\author{Md Kashif Jamal}
\address{Md Kashif Jamal \\ Department of Mathematics \\
Indian Institute of Technology Indore \\
Simrol,  Indore,  Madhya Pradesh 453552, India.} 
\email{kashif.sxcr@gmail.com,  kashif.jamal@iiti.ac.in}

 \author{Bibekananda Maji}
\address{Bibekananda Maji\\ Department of Mathematics \\
Indian Institute of Technology Indore \\
Simrol,  Indore,  Madhya Pradesh 453552, India.} 
\email{bibek10iitb@gmail.com,  bibekanandamaji@iiti.ac.in}


\maketitle
\begin{center}
{\it Dedicated to Srinivasa Ramanujan on his $134$th birth anniversary}
\end{center}

\begin{abstract}
One of the celebrated formulas of Ramanujan is about odd zeta values,  which has been studied by many mathematicians over the years.  A notable extension was given by Grosswald in 1972.  Following  Ramanujan's idea,   we rediscovered a Ramanujan-type identity for $\zeta(2k+1)$ that was first established by Malurkar and later by Berndt using different techniques.  In the current paper,  we extend the aforementioned identity of Malurkar and Berndt to derive a new Ramanujan-type identity for $L(2k+1, \chi_1)$, where $\chi_1$ is the principal character modulo prime $p$.  In the process,  we encounter a new family of Ramanujan-type polynomials  and we notice that a particular case of these polynomials has been studied by Lal\'{i}n and Rogers in 2013.  Furthermore,  we establish a character analogue of Grosswald's identity and a few more interesting results inspired from the work of Gun,  Murty and Rath. 

\end{abstract}

\section{introduction}
The Riemann zeta function $\zeta(s)$ is one of the most important special functions in number theory and its theory plays a crucial role for the development of analytic number theory.  In 1734,  Euler established an exact formula for $\zeta(2k)$ in terms of powers of $\pi$ and Bernoulli number $B_{2k}$.  Mainly,  he showed that,  for any natural number $k$, 
	\begin{equation}  \rz{2k}= (-1)^{k+1}\frac{(2 \pi)^{2k} B_{2k}}{2 (2k)!}.    \label{euler's identity} 
\end{equation}
This formula instantly tells us that all even zeta values are transcendental since $\pi$ is transcendental and Bernoulli numbers are rational.  However,  in the literature no such simple explicit formula for $\zeta(2k+1)$ exists and not much is known about the algebraic nature of odd zeta values.  Stunningly,  Roger Apr\'{e}y \cite{apery1,apery2}, in 1959,  proved that $\zeta(3)$ is an irrational number,  but the arithmetic nature of $\zeta(3)$ is not known yet.  
In 2001,  Zudilin \cite{zudilin} proved that atleast one of $\zeta(5)$, $\zeta(7)$, $\zeta(9)$, and $\zeta(11)$ is irrational.  Around the same time,  Rivoal \cite{rivoal},  Ball and Rivoal \cite{ballrivoal} proved that infinitely many odd zeta values are irrational.  These are the current best results in this direction.

Ramanujan's  notebook and lost-notebook contain many intriguing identities and one of the most celebrated identities is about $\zeta(2k+1)$.  The identity is the following: 
\begin{align}
     (4\alpha)^{-k} \left\{ \frac{1}{2} \zeta(2k+1) + \sum_{n=1}^{\infty} \frac{n^{-2k-1}}{e^{2 \alpha n}-1} \right\} &- (-4 \beta)^{-k}  \left\{ \frac{1}{2} \zeta(2k+1) + \sum_{n=1}^{\infty} \frac{n^{-2k-1}}{e^{2 \beta n}-1} \right\} \nonumber \\
     &= \sum_{j=0}^{k+1} (-1)^{j-1} \frac{B_{2j}}{(2j)!} \frac{B_{2k+2-2j}}{(2k+2-2j)!} \alpha^{k+1-j} \beta^{j}, \label{RamaF}
 \end{align}
where $\alpha \beta = \pi^{2}$ with $\alpha, \beta  >0$, and $k \in \Z \backslash \{0\}$.  This identity can be found in Ramanujan's second notebook \cite[p.~173,  Ch.~14,  Entry 21(i)]{rn}  as well as in  the lost notebook \cite[p.~319,  Entry (28)]{lnb}.  In the course of time,  this identity took attention of many mathematicians.  
To know more about this formula  readers are encouraged to see the paper of Berndt and Straub \cite{Mainpaper}.  Recently Dixit and the third author \cite{DM20} established a beautiful generalization of \eqref{RamaF} while extending an identity of Kanemitsu et al. \cite{KTY01} and another generalization to the Hurwitz zeta function can be found in Dixit et. al. \cite{DGKM20}. 
Character analogues of Ramanujan's formula  have been studied by Berndt \cite{Berndt73,  Berndt75} and Bradley \cite{bradley}. 
 Further,  various generalizations of \eqref{RamaF} have been made by many mathematicians,  readers are refer to see \cite{BanerjeeKumar21,  Berndt77, Mainpaper, CCVW21, GuptaMaji}.  

Now setting $\a=\b=\pi$ and replacing $k$ by $2k+1$ in  \eqref{RamaF},  one can obtain the following Lerch's formula for $\zeta(4k+3)$, 
\begin{align}
    \zeta(4k+3) = \pi^{4k+3} 2^{4k+2} \sum_{j=0}^{2k+2} (-1)^{j+1} \frac{B_{2j}}{(2j)!} \frac{B_{4k+4-2j}}{(4k+4-2j)!} - 2 \sum_{n=1}^{\infty}\frac{1}{n^{4k+3}(e^{2\pi n}-1)}. \label{lerch_f}
\end{align}
To obtain \eqref{RamaF},  Ramanujan used the partial fraction expansion of $\cot ( \sqrt{w\alpha}) \coth (\sqrt{w \beta})$,  unfortunately he made an error.  He \cite[p. 171, Ch. 14, Entry 19(i)]{rn},  \cite[p. 318,  Entry (21)]{lnb} offered the below partial fraction decomposition 
\begin{equation} 
    \pi^2 xy \cot(\pi x)\coth(\pi y) = 1 + 2 \pi xy\summ{n=1}{\infty}{\frac{n \coth{ \big(  \frac{\pi n x }{y} \big) } }{n^2 + y^2} -2 \pi xy \summ{n=1}{\infty}{\frac{n \coth{ \big( \frac{\pi n y}{x} \big) }}{n^2-x^2}}  }, \label{Rama Cot In}
\end{equation} 
in which the two infinite series on the right hand side diverge individually.  The corrected version of the above partial fraction formula was later established by R. Sitaramchandrarao \cite{sitaram} as below:
 \begin{align}
     \pi^{2} xy \cot (\pi x) \coth{ (\pi y ) }  & = 1+ \frac{\pi^{2}}{3} ( y^{2}-x^{2}) \nonumber \\
     & - 2 \pi xy \sum_{m=1}^{\infty} \left( \frac{y^{2} \coth{  \big( \frac{\pi m x}{y} \big) }}{m(m^{2}+y^{2} ) } + \frac{x^{2} \coth{ \big( \frac{\pi m y}{x} \big) }}{m(m^{2}-x^{2})} \right).
 \end{align}
On replacing $\pi x$ by $ \sqrt{w \alpha }$ and $\pi y$ by $ \sqrt{w \beta}$,  followed by some elementary calculations and comparing the coefficients of $w^{k}$, for $k \geq 1$, identity $\eqref{RamaF}$ can be easily obtained. The first published proof of \eqref{RamaF} was given by Malurkar \cite{Malurkar25} in 1925. 
One of the notable extensions of the Ramanujan's formula \eqref{RamaF} was given by Emil Grosswald \cite{grosswald} in 1972.  For any $z$ lying in the upper half plane $\mathbb{H}$, 
 \begin{align}
     F_{2k+1}(z)- z^{2k} F_{2k+1} \left(-\frac{1}{z} \right) &= \frac{1}{2} \zeta(2k+1) ( z^{2k}-1) \nonumber \\
     &  + \frac{(2 \pi i )^{2k+1}}{2z} \sum_{j=0}^{k+1} z^{2k+2-2j} \frac{B_{2j}}{(2j)!} \frac{B_{2k+2-2j}}{(2k+2-2j)!}, \label{grosswald_eq}
 \end{align}
where $F_{k}(z)= \sum_{n=1}^{\infty} \sigma_{-k}(n) e^{2 \pi i  n z}$ with $\sigma_{-k}(n)= \sum\limits_{ d|n  } d^{-k}$.  Setting $z= i \beta / \pi $,  $\alpha \beta = \pi^2$,  with $\a, \b >0$,  Grosswald's identity immediately gives Ramanujan's identity \eqref{RamaF} for odd zeta values.  More importantly,  the above identity establishes a connection with the Eisenstein series $E_{2k}(z)$ for the full modular group $SL_2(\mathbb{Z})$ since the Fourier series expansion suggests that,  for any $k \geq 2$, 
\begin{align*}
E_{2k}(z) = 1- \frac{4k}{B_{2k}} F_{1-2k}(z). 
\end{align*}
Again,  for $k\geq 3$ odd,  we can think $F_k(z)$ as an Eichler integral of the first kind.  To know more about this connection,  readers are encouraged to see \cite{BerndtStraub16,  goldstein,GunMP}. 
The polynomial on the right hand side of \eqref{grosswald_eq} has been studied by Gun,  Murty and Rath \cite{GunMP} and they termed it as \emph{Ramanujan's polynomial},  namely,  
 \begin{equation}\label{Ramanujan polynomial}
     R_{2k+1}(z)=\sum_{j=0}^{k+1} z^{2k+2-2j} \frac{B_{2j}}{(2j)!} \frac{B_{2k+2-2j}}{(2k+2-2j)!}. 
 \end{equation}
Murty, Smith and Wang \cite{ZofR}  proved that  $R_{2k+1}(z)$ has only unimodular roots with multiplicity $1$ apart from four real roots, for $k \geq 4$.  
 More specifically, they have shown that the only zeros of $R_{2k+1}(z)$  that are roots of unity  at $\pm i$ which occur only if $ k$ is an even natural number, and $ \pm \rho, \pm \rho^{2}$ with
$\rho= e^{2 \pi i/3}$ if and only if $3$ divides $k$. 
This indicates that there are zeros of $R_{2k+1}(z)$ that are lying  in $\mathbb{H}$ and not $2k$-th roots of unity.  Therefore,  Grosswald's identity $\eqref{grosswald_eq}$ leads to 
\begin{align}
    \zeta(2k+1)=\frac{2}{z^{2k}-1} \left(F_{2k+1}(z) - z^{2k} F_{2k+1} \left(- \frac{1}{z}\right) \right)
\end{align}
for some $z \in \mathbb{H} \cap \overline{\mathbb{Q}}$.  The above observation inspired Gun, Murty and Rath \cite{GunMP}  to study the nature of the special values of the function $G_{2k+1}(z)$ defined as
\begin{align}
    G_{2k+1}(z):=\frac{1}{z^{2k}-1} \left(F_{2k+1}(z) - z^{2k} F_{2k+1} \left(- \frac{1}{z}\right) \right).
\end{align}
They obtained the following result. 

Let $k \in \mathbb{N} \cup \{ 0\}$.  Define $\delta_k = 0, 1, 2, 3$ accordingly as the
gcd$(k, 6)$ equals $1, 2, 3$ or $6$, respectively. 
For any algebraic  $z \in \mathbb{H}$, the quantity
\begin{equation*}
F_{2k+1}(z) - z^{2k}F_{2k+1}\left( - \frac{1}{z}\right)
\end{equation*}
is transcendental with at most $2k +2+ \delta_k$ exceptions. 


Ramanujan's \cite[p. 171, Ch. 14, Entry 19(i)]{rn},  \cite[p. 318,  Entry (21)]{lnb} main idea was to use the partial fraction decomposition \eqref{Rama Cot In} to derive the formula \eqref{RamaF} for $\zeta(2m+1)$.  Subsequent to that on the same page,  he \cite[p. 171, Ch. 14, Entry 19(iii)]{rn} offered the following partial fraction decomposition for product of two tangent functions: \\
\emph{
Let $x$ and $y$ be complex numbers such that $y/x$ is not purely imaginary. Then
\begin{align}
    \frac{\pi}{4} \tan \left(\frac{\pi x}{2} \right) \tanh \left(\frac{\pi y}{2} \right)= y^{2}  \summ{n=0}{\infty}{\frac{ \tanh\left((2n+1) \frac{\pi  x}{ 2y}\right)}{(2n+1)\{(2n+1)^2 + y^2\}} + x^{2}\summ{n=0}{\infty}{\frac{  \tanh\left((2n+1) \frac{\pi  y}{ 2x}\right)}{(2n+1) \{(2n+1)^2-x^2 \}}}  }. \label{Ram Tan Id}
\end{align} }
In the next section,  we state all the main results on this paper.  
\section{Main results}
Quite surprisingly,  the above formula \eqref{Ram Tan Id} gives us a  Ramanujan-type identity for $\zeta(2k+1)$.     
\begin{theorem}\label{thm_analogue_ramanujan}
Let $\alpha$ and $\beta$ be two positive real numbers such that $ \alpha \beta = \frac{\pi^{2}}{4}$.  For any integer $k \geq 1$,  we have
\begin{align}
   & (4\alpha)^{-k}\left(\frac{1}{2} \zeta(2k+1)\bracket{1-2^{-2k-1}} - \sum_{n=0}^{\infty} \frac{(2n+1)^{-2k-1}}{(e^{2(2n+1)\alpha}+1) }\right)  \nonumber\\
     &-(-4\beta)^{-k} \left( \frac{1}{2} \zeta(2k+1)\bracket{1-2^{-2k-1}} -  \sum_{n=0}^{\infty}\frac{(2n+1)^{-2k-1}}{(e^{2(2n+1)\beta}+1) } \right) \nonumber \\
     & = \sum_{j=1}^{k}  (-1)^{j-1} (2^{2j}-1)  (2^{2k+2-2j}-1)\frac{B_{2j}}{(2j)!} \frac{B_{2k+2-2j}}{(2k+2-2j)! } \alpha^{k+1-j}\beta^{j} \label{eq_analogue_ramanujan}.
 \end{align}
\end{theorem}

\begin{remark} This identity was first proved by Malurkar \cite{Malurkar25},  although his derivation was completely based on complex analysis technique.  In 1978,  Berndt \cite[Theorem 4.7]{Berndt78} also recovered the same identity while studying a generalized Eisenstien series.   
Recently,  Dixit and Gupta \cite[Corollary 4.2]{DG-21} have also obtained the above identity,  while studying Koshliakov zeta function.  
Here we would like to emphasise that the above theorem is not only true for positive integers $k$,  but also true for negative integers $k$.  In the later case,  one must consider the finite sum involving Bernoulli numbers as an empty sum.  
\end{remark}

\begin{remark} An alternative form of the above result is the following identity: 
    \begin{align}
   & (4 \alpha)^{-k} \sum_{n=0}^{\infty} \frac{\tanh{((2n+1) \alpha})}{(2n+1)^{2k+1}} - ( -4  \beta)^{-k}\sum_{n=0}^{\infty} \frac{\tanh{((2n+1) \beta})}{(2n+1)^{2k+1}}   \nonumber \\
     &  =  2 \sum_{j=1}^{k}  (-1)^{j-1} (2^{2j}-1)  (2^{2k+2-2j}-1)\frac{B_{2j}}{(2j)!} \frac{B_{2k+2-2j}}{(2k+2-2j)! } \alpha^{k+1-j}\beta^{j}. \label{Ramana}
 \end{align}
 \end{remark}
 As an immediate implication of \eqref{Ramana},  we obtain an exact evaluation of an infinite series associated to tan-hyperbolic function. 
\begin{corollary}\label{tan particular case}
For $k \geq 0$, we have
 \begin{align}
    \sum_{n=0}^{\infty} \frac{\tanh \left((2n+1) \frac{\pi}{2} \right)}{(2n+1)^{4k+3}} &= \frac{\pi^{4k+3}}{2} \sum_{j=1}^{2k+1}  (-1)^{j-1} (2^{2j}-1)  (2^{4k+4-2j}-1)\frac{B_{2j}}{(2j)!} \frac{B_{4k+4-2j}}{(4k+4-2j)! }.
    \label{Tan particular case}
 \end{align}
 \end{corollary}
Special cases $k=0$ and $ k=1$  of \eqref{Tan particular case} are noted down by Ramanujan  as Entry $25$  $(iii)$, $(iv)$ in Chapter $14$ of his second notebook \cite[p.~ 180]{rn}.  
The above identity has been also proved by many mathematicians and to know more about this identity we refer to \cite{Berndt78}.  
This identity immediately implies that the above infinite series converges to a transcendental number.  More generally,  Berndt \cite[Theorem 4.11]{Berndt78} proved that $\sum_{n=0}^{\infty} \frac{\tanh \left((2n+1) \frac{\pi \theta}{2} \right)}{(2n+1)^{2k+1}} $ is a transcendental number,  when $\theta$ is some certain quadratic irrational number.  Similar behaviour for other trigonometric Dirichlet series have been studied by Lal\'{i}n et al.  \cite{LRR14} and Straub \cite{Straub16}.  
The identity \eqref{Tan particular case} also provides a beautiful formula for $\zeta(4k+3)$ analogous to  Lerch's famous identity \eqref{lerch_f}.  For any $k \geq 0$,  
\begin{align}
      \zeta(4k+3)\left(1-\frac{1}{2^{4k+3}} \right) &= \frac{\pi^{4k+3}}{2} \sum_{n=1}^{2k+1}  (-1)^{n+1} (2^{2n}-1)  (2^{4k+4-2n}-1)\frac{B_{2n}}{(2n)!} \frac{B_{4k+4-2n}}{(4k+4-2n)! } \nonumber \\
      & + 2\sum_{n=0}^{\infty} \frac{(2n+1)^{-4k-3}}{(e^{(2n+1)\pi}+1)}.
      \label{eq_analogue_lerch}
\end{align} 

Now we would like to highlight that Theorem  \ref{thm_analogue_ramanujan} can also  be derived using contour integration technique and while doing so we realized that a more general version of it is indeed true.  Here we state a one-variable generalization of Theorem $\ref{thm_analogue_ramanujan}$.

\begin{theorem} \label{thm_gen_rama_analogue}
Let $p$ be a prime number and $\al$, $\be > 0$ such that $\al \be = \frac{\pi^2}{p^2}$.  Let $\chi_1$ denote the principal Dirichlet character modulo $p$.  
 Then for any $k \in \mathbb{Z} \backslash \{0\}$,  we have
 \begin{align}
    &(4 \al)^{-k}\bracket{  \frac{p-1}{2}L(2k+1,\chi_1)- \summ{n=1}{\infty}{a_n \left( \sum_{d|n} \frac{\chi_1(d)}{d^{2k+1}} \right) }e^{-2n \al}  } \nonumber \\
    &  -(-4 \be)^{-k}\bracket{  \frac{p-1}{2}L(2k+1,\chi_1)- \summ{n=1}{\infty}{a_n \bracket{ \sum_{d|n} \frac{\chi_1(d)}{d^{2k+1}} } }e^{-2n \be}  } \nonumber \\
    &=\summ{j=1}{k}  {(-1)^{j-1}(p^{2j}-1)(p^{2k+2-2j}-1)\frac{B_{2j}}{(2j)!}\frac{B_{2k+2-2j}}{(2k+2-2j)!}\al^{k+1-j}\be^j }, \label{eqn_gen_rama_analogue}
\end{align}
 where $L(s, \chi)$ is the Dirichlet $L$-function and
 \begin{align}\label{a_n}
 a_n = \begin{cases} 1, & \text{if}\,\, \gcd(n,p)=1,  \\
                    1-p,  & \text{if} \,\, \gcd(n,p)=p.
 \end{cases}
 \end{align}
\end{theorem}
Letting $\al = \be = \frac{\pi}{p}$ and replacing $ k\to 2k+1 $ as an odd positive integer, Theorem \ref{thm_gen_rama_analogue}  immediately provides an interesting identity for $L(4k+3,\chi_1)$ analogous to Lerch's formula \eqref{lerch_f} for $\zeta(4k+3)$. 
\begin{corollary} \label{cor_gen_lerch_analogue}
Let $p$ be a prime number and $\chi_1$ be the principal Dirichlet character modulo $p$.  
For any integer $k \geq 0$,  we have
\begin{align}
    L(4k+3,\chi_1) & = \frac{2}{p-1}\summ{n=1}{\infty}{a_n \bracket{ \sum_{d|n}{\frac{\chi_1(d)}{d^{4k+3}}}e^{-\frac{ 2 n \pi}{p}} }} \nonumber \\
    &+ \frac{2^{4k}}{2(p-1)}  \left( \frac{\pi}{p} \right)^{4k+3}  \summ{j=1}{k}{(-1)^{j-1}(p^{2j}-1)(p^{4k+4-2j}-1)\frac{B_{2j}}{(2j)!} \frac{B_{4k+4-2j}}{(4k+4-2j)!} }. \label{eqn_gen_lerch_analogue}
\end{align}
\end{corollary}

\begin{corollary}\label{for k negative}
Let $\a,  \b>0 $ such that $\a \b = \pi^2/p^2$.  For any integer $k \geq 0$,  we have 
\begin{align*}
 \alpha^{k+1} \summ{n=1}{\infty} {a_n \bigg( \sum_{d|n} {\chi_1(d) d^{2k+1}} \bigg) }e^{-2n \al} &  - (-\beta)^{k+1} \summ{n=1}{\infty}{a_n \bigg( \sum_{d|n} {\chi_1(d) d^{2k+1}} \bigg) }e^{-2n \beta} \nonumber \\ 
&= (p-1) (p^{2k+1}-1) \left( \alpha^{k+1} -  (-\beta)^{k+1} \right)  \frac{B_{2k+2}}{4k+4}. 
\end{align*}
\end{corollary}
One can observe that the above identity is a resemblant of the following identity of Ramanujan  \cite[Vol.~1, p.~259]{ramanote}.

{\emph {Let $k \geq 0$ be a positive integer.  For $\alpha,\beta>0$ with $\alpha \beta=\pi^2$,  we have
\begin{align}\label{Rama_k=1_ r=2m+1}
\alpha^{k+1} \sum_{n=1}^{\infty} \frac{n^{2k+1}}{e^{2n\alpha}-1}-(-\beta)^{k+1}  \sum_{n=1}^{\infty} \frac{n^{2k+1}}{e^{2n\beta}-1}=(\alpha^{k+1} -(-\beta)^{k+1}) \frac{B_{2k+2}}{4k+4}.
\end{align}}}

The point to be noted that the Theorem \ref{thm_gen_rama_analogue} is not true for $k=0$ since $L(s, \chi_1)$ has a simple pole at $s=1$.  Corresponding to $k=0$, we obtain the below identity.  
\begin{theorem} \label{thm_k=0}
Let us consider the variables as we defined in Theorem \ref{thm_gen_rama_analogue}.  Then
    \begin{equation}
        \summ{n=1}{\infty}{a_n \bigg( \sum_{d|n}{\frac{\chi_1(d)}{d}}} e^{-2n \al } \bigg) -\summ{n=1}{\infty}{a_n \bigg( \sum_{d|n}{\frac{\chi_1(d)}{d}}} e^{-2n \be } \bigg) = \frac{(p-1)^2}{2p}\log\bracket{\frac{\pi}{p \al}}. \label{identity_k_0}
    \end{equation}
\end{theorem}
In particular,  for $p=2$,  we derive an interesting identity.  
\begin{corollary} For $\alpha,  \beta >0$ with $\alpha \beta = \pi^2/4$,  we get
    \begin{equation}\label{analogous to Dedekind}
        \summ{n=0}{\infty}{\frac{(2n+1)^{-1}}{ e^{2(2n+1)\al}+1 } } - \summ{n=0}{\infty}{ \frac{(2n+1)^{-1}}{e^{2(2n+1)\be}+1 } } =\frac{\log(\pi)- \log(2 \al)}{4}.
    \end{equation}
\end{corollary}
\begin{remark}
The identity \eqref{analogous to Dedekind} is clearly analogous to the following well-known identity of Ramanujan \cite[Ch. 16, Entry 27(iii)]{ramanote} \cite[p.~320, Formula (29)]{lnb}.
For $\alpha ,\beta >0$ and $\alpha\beta=\pi^2 $,
\begin{equation}\label{Dedekind}
\sum_{n=1}^{\infty}\frac{1}{n(e^{2n\alpha}-1)}-\sum_{n=1}^{\infty}\frac{1}{n(e^{2n\beta}-1)}=\frac{\beta-\alpha}{12} + \frac{1}{4}\log\left(\frac{\alpha}{\beta}\right),
\end{equation}
which is equivalent to the transformation formula for the logarithm of the Dedekind eta function $\eta(z)$,  a half-integral weight modular form.  To know more about it, readers can see \cite[p.~256]{bcbramsecnote}, \cite[p.~43]{bcbramthinote}.  It would be interesting to find such an equivalent form for the identity \eqref{analogous to Dedekind}. 
\end{remark}



\subsection{Analogue of Grosswald's identity and Ramanujan-type polynomials}
In the introduction,  we have already seen that the Grosswald's identity \eqref{grosswald_eq} is a natural generalization of Ramanujan's formula \eqref{RamaF}.  The generalized divisor function $\sigma_{-k}(n) = \sum_{d|n} d^{-k}$ appeared in Grosswald's identity.  A natural character analogue of the generalized divisor function is defined by $\sigma_{-k, \chi}(n) = \sum_{d|n} \chi(d) d^{-k}$,  where $\chi$ is a Dirichlet character.  This character analogue of the divisor function has been studied by many mathematicians.  Now we state a generalization Theorem \ref{thm_gen_rama_analogue},  namely,  a Grosswald-type identity for complex variable $z$.

\begin{theorem} \label{thm_gen_grosswald_analogue}

Let $p$ be a prime number and $\chi_1$ be the principal Dirichlet character modulo $p$.  For any $k\in \mathbb{N}$ and $z \in \mathbb{H}$,  we have

\begin{align}
 \left( z p\right)^{2k} &  \mathfrak{F}_{2k+1,  \chi_1}\left(-\frac{1}{p^2 z} \right)  - \mathfrak{F}_{2k+1, \chi_1}(z) =   \frac{p-1}{2}L(2k+1,\chi_1)  \left\{ \left( z p\right)^{2k} -1  \right\}   \nonumber \\
   + & \frac{(2\pi i)^{2k+1}}{ 2z \, p^{2k+2}} \summ{j=1}{k}  {(p^{2j}-1)(p^{2k+2-2j}-1)\frac{B_{2j}}{(2j)!}\frac{B_{2k+2-2j}}{(2k+2-2j)!} } ( pz)^{2k+2-2j},\label{eqn_gen_grosswald_analogue}
\end{align}
 where 
\begin{align}\label{F-function}
 \mathfrak{F}_{k,\chi_1}(z):= \sum_{n=1}^\infty a_n \sigma_{-k,  \chi_1}(n) \exp( 2\pi i n z).
 \end{align}
\end{theorem}
Here we can clearly observe that the polynomial present on the right hand side of the above theorem is an analogue of the Ramanujan's polynomial \eqref{Ramanujan polynomial}.  For any prime number $p$ and $k \in \mathbb{N}$,   we define this Ramanujan-type polynomial as
\begin{align}\label{new Ramanujan-type polynomila}
\mathfrak{R}_{2k+1,p}(z):=  \summ{j=1}{k}  {(p^{2j}-1)(p^{2k+2-2j}-1)\frac{B_{2j}}{(2j)!}\frac{B_{2k+2-2j}}{(2k+2-2j)!} } (p z)^{2k+2-2j}. 
\end{align}
Interestingly,  the above polynomial for $p=2$ has been studied by Lal\'in and Rogers \cite[Theorem 2]{ML}.  Utilizing their result,  we can immediately say that all the complex zeros of the  polynomial $\mathfrak{R}_{2k+1,2}(z)$ lie on the circle $|z|=1/2$ and roots are simple.  Next,  we state a more general conjecture, which suggests that a more general version of Lal\'in and Rogers's result might be true.
\begin{conjecture}\label{conjecture1}
Let $p$ be a prime number and $k$ be a natural number.  The Ramanujan-type polynomial $\mathfrak{R}_{2k+1,p}(z)$ has only real zero $z=0$ of multiplicity $2$,  and rest are the non real zeros on the circle $|z| =1/p$ and are simple.  
\end{conjecture}
We state another conjecture,  which gives information about the zeros of the Ramanujan-type polynomial $\mathfrak{R}_{2k+1,p}(z)$ that are $2k$th roots of unity. 
\begin{conjecture}\label{conjecture2}
If $2|k$, then the polynomials $\mathfrak{R}_{2k+1,p}(z)$ and $(pz)^{2k}-1$ have only common roots $z=\pm i/p$, and if $2\nmid k$,  they do have any common roots.
\end{conjecture}
In the final section,  we provide numerical evidences for these two conjectures. 
Now assuming these conjectures,  inspired from the work of Gun,  Murty and Rath \cite{GunMP},  we obtain the below mentioned results.  
First,  let us define a quantity,  extracted from Theorem \ref{thm_gen_grosswald_analogue}, for any $z\in \mathbb{H}$ with $(pz)^{2k} \neq 1$,  as
\begin{align}\label{quantity_G}
\mathfrak{G}_{2k+1, \chi_1}(z):= \frac{   \left( z p\right)^{2k}   \mathfrak{F}_{2k+1,  \chi_1}\left(-\frac{1}{p^2 z} \right)  - \mathfrak{F}_{2k+1, \chi_1}(z) }{ (pz)^{2k}-1}.
\end{align}

\begin{proposition}\label{thm_gmr_one_alg}
The set 
$$ \{  \mathfrak{G}_{2k+1,\chi_1}(z) | \hspace{.1cm} z \in \mathbb{H} \cap \overline{ \mathbb{Q}}, \hspace{.2cm} (pz)^{2k} \neq 1 \} $$
contains at most one algebraic number. 
\end{proposition}

\begin{theorem}\label{thm_gmr_exception}
Let $k \geq 0$ be an integer and $p$ be a prime.  Define $\delta_k=0,1$ respectively depending on if $\gcd(k,2)$ equals to $1$, or $2$. Then, with at most $2k+ \delta_k$ exceptions,  the quantity 
\begin{align}
\left( z p\right)^{2k}   \mathfrak{F}_{2k+1,  \chi_1}\left(-\frac{1}{p^2 z} \right)  - \mathfrak{F}_{2k+1, \chi_1}(z)
\end{align}
is transcendental for every algebraic number $z \in \mathbb{H}$.  This suggests that,  there are at most $2k+ \delta_k$ algebraic numbers $z$ on the upper half plane such that $\mathfrak{F}_{2k+1, \chi_1}(z)$ and $ \mathfrak{F}_{2k+1,  \chi_1}\left(-\frac{1}{p^2 z} \right)$ are both algebraic.  
\end{theorem}

\section{Preliminaries}

The Riemann zeta function $\zeta(s)$ obeys  the following asymmetric form of functional equation,  namely,  
\begin{equation}\label{asymmetric_functional_zeta}
    \zeta(s)= 2^{s} \pi^{s-1} \zeta(1-s) \Gamma(1-s) \sin \left( \frac{\pi s}{2}\right). 
\end{equation}
The Laurent series expansion of $\Gamma(s)$ around the point $s=0$ is the following
\begin{align}
    \Gamma(s)= \frac{1}{s} -\gamma + \frac{1}{2} \left( \gamma^{2} + \frac{\pi^{2}}{6}\right) s- \frac{1}{6} \left( \gamma^{3} + \gamma \frac{\pi^{2}}{2} + 2 \zeta(3) \right) s^{2} + \cdots,
\end{align}
where $\gamma$ is the well-known Euler-Mascheroni constant.  The gamma function $\Gamma(s)$ satisfies the below reflection formula,  namely, 
\begin{align}\label{reflection}
    \Gamma(s) \Gamma(1-s) = \frac{\pi}{ \sin (\pi s)}, \hspace{.5cm} \forall s \in \C \backslash \Z.
\end{align}
The next result provides an important information about the asymptotic growth of $\Gamma(s)$.  For $s=\sigma+iT$, in a vertical strip $\alpha \leq \sigma \leq \beta$,
\begin{equation}\label{Stirling's formula}
    | \Gamma ( \sigma+iT) | = \sqrt{2 \pi} |T|^{\sigma- 1/2} e^{-\frac{1}{2} |T|} \left(1+ \mathcal{O}\left( \frac{1}{|T|}\right)\right), \quad \textrm{as} \,\,  |T| \to \infty.
\end{equation}
Now we will state a few results which will be essential in proving our main results.

\begin{lemma}\label{Use of functional equation}
For any $k \in \mathbb{Z}$,  one has
    \begin{equation}
        \Gamma(s)\zeta(s)\zeta(s+2k+1) = (-1)^{k}(2 \pi)^{2s+2k}\Gamma(-s-2k)\zeta(1-s)\zeta(-s-2k). \label{lem_feq}
    \end{equation}
    \end{lemma}
\begin{proof}
Employing the functional equation \eqref{asymmetric_functional_zeta} of $\zeta(s)$ and the reflection formula \eqref{reflection} for $\Gamma(s)$,  one can obtain this identity.  
\end{proof}
\begin{lemma} \label{lem_zeta_l_sum}
Let $p$ be a prime number and $a_n$ be the function as we defined in \eqref{a_n}.
Let $\chi_1$ be the principal Dirichlet character $mod$ $p$.  Then for $ \Re(s) > \max\{ 1,  -2k \}$, 
    \begin{equation}
        \summ{n=1}{\infty}{a_n \bracket{\sum_{d|n}\frac{\chi_1(d)}{d^{2k+1}}}}\frac{1}{n^s} = L(s+2k+1,\chi_1)\zeta(s)(1-p^{1-s}).
    \end{equation}
\end{lemma}
\begin{proof}
First,  we note that $\zeta(s) (1- p^{1-s})= \sum_{n=1}^{\infty} a_n  n^{-s}$ for $\Re(s)>1$.  
Using Dirichlet convolution,  for $ \Re(s) > \max\{ 1,  -2k \}$,  one can verify that $$ L(s+2k+1,\chi_1)\zeta(s)(1-p^{1-s}) = \sum_{n=1}^\infty \left( \frac{ a_{\frac{n}{d}}  \chi_1(d) }{d^{2k+1} }  \right) \frac{1}{n^s}.$$
Note that as $\chi_1(d)$ survives only when $\gcd(d, p)=1$,  and in that case, we can check that $a_{\frac{n}{d}}= a_n$.  This completes the proof of this lemma.  
\end{proof}

\section{Proof of main results}

\begin{proof}[Theorem {\rm \ref{thm_analogue_ramanujan} }][]
Let $\alpha$ and  $\beta$ be positive number satisfying $\alpha \beta = \frac{\pi^2}{4}$.  
Substituting $\frac{\pi x}{2} = \sqrt{w \a}$,  $\frac{\pi y}{2}= \sqrt{w \beta}$ in the partial fraction decomposition \eqref{Ram Tan Id},  we see that
\begin{align}
    \frac{\pi}{4} \tan \left(\sqrt{ w \alpha} \right) \tanh \left(\sqrt{w \beta} \right)& = w\sum_{n=0}^{\infty} \frac{\tanh \{(2n+1) \alpha \}}{(2n+1) \{ \alpha(2n+1)^{2} + w\} } \nonumber\\
    &  + w\sum_{n=0}^{\infty} \frac{\tanh \{(2n+1) \beta  \}}{(2n+1) \{ \beta (2n+1)^{2} - w \} }. \label{ptan}
\end{align}
Now we can write
\begin{align}
w \sum_{n=0}^{\infty} \frac{ \tanh \{(2n+1) \alpha \} }{(2n+1) \{ \alpha(2n+1)^{2} + w\} }  & = w \sum_{n=0}^{\infty} \frac{  \tanh \{(2n+1) \alpha \} }{\alpha (2n+1)^3 } \left( 1 + \frac{w}{\alpha(2n+1)^2}  \right)^{-1}  \nonumber \\
& = \sum_{k=1}^{\infty} \sum_{n=0}^\infty \frac{ \tanh \{(2n+1) \alpha \}  }{ (2n+1)^{2k+1}} \frac{(-1)^{k+1} w^{k}}{\alpha^{k}}. \label{series expansion1}
\end{align}
Similarly,  one can show that
\begin{align}
w\sum_{n=0}^{\infty} \frac{\tanh \{(2n+1) \beta  \}}{(2n+1) \{ \beta (2n+1)^{2} - w \} } = \sum_{k=1}^{\infty} \sum_{n=0}^\infty \frac{ \tanh \{(2n+1) \beta \}  }{ (2n+1)^{2k+1}} \frac{ w^{k}}{\beta^{k}}. \label{series expansion2}
\end{align}
In view of \eqref{series expansion1} and \eqref{series expansion2},  the coefficient of $w^{k}$ of the right hand side expression of \eqref{ptan} is 
\begin{align}\label{coefficient of w^k_RHS}
 \sum_{n=0}^\infty \frac{ \tanh \{(2n+1) \alpha \}  }{ (2n+1)^{2k+1}} \frac{(-1)^{k+1} }{\alpha^{k}} + \sum_{n=0}^\infty \frac{ \tanh \{(2n+1) \beta \}  }{ (2n+1)^{2k+1}} \frac{1 }{\beta^{k}}.
\end{align}
Now we mention the Laurent series expansions \cite[p.~5]{Cohen} for $\tan z$ and $\tanh z$ around $z=0$,  i.e.,  for $0 < |z| < \frac{\pi}{2}$,  
\begin{align*}
\tan z & = \sum_{j=1}^{\infty} (-1)^{j-1} 2^{2j} (2^{2j}-1) \frac{B_{2j} }{(2j)!} z^{2j-1},  \quad
\tanh z  = \sum_{n=1}^\infty 2^{2n} (2^{2n}-1) \frac{B_{2n}}{(2n)!} z^{2n-1}.   
\end{align*}
On substituting the above series expansions,  the left hand side of \eqref{ptan} becomes
\begin{align*}
    \frac{\pi}{4} \left(\sum_{j=1}^{\infty}   (-1)^{j-1} 2^{2j}  (2^{2j}-1)\frac{B_{2j}}{(2j)!} (w\alpha)^{j-\frac{1}{2}} \right)
  \left( \sum_{n=1}^{\infty}  2^{2n}  (2^{2n}-1)\frac{B_{2n}}{(2n)!}  (w \beta)^{n- \frac{1}{2}}\right).
\end{align*}
Replacing $n+j-1$ by $k$,  the above product can be re-written as
\begin{align}\label{LHS_coefficient w^k}
  2^{2k+1}  \sum_{k=1}^\infty   \sum_{j=1}^{k+1} (-1)^{j-1}  (2^{2j}-1)  (2^{2k+2-2j}-1) \frac{B_{2j}}{(2j)!}  \frac{B_{2k+2-2j }}{(2k+2-2j)!} \alpha^j \beta^{k+1-j} w^{k}. 
\end{align}
Now collecting the coefficient of $w^k$ from \eqref{LHS_coefficient w^k} and together with \eqref{coefficient of w^k_RHS},  we arrive at
\begin{align}\label{alternative form_Ramanujan}
 &\sum_{n=0}^\infty  \frac{ \tanh \{(2n+1) \alpha \}  }{ (2n+1)^{2k+1}} \frac{(-1)^{k+1} }{\alpha^{k}} + \sum_{n=0}^\infty \frac{ \tanh \{(2n+1) \beta \}  }{ (2n+1)^{2k+1}} \frac{1 }{\beta^{k}} \nonumber \\ &  =  2^{2k+1}   \sum_{j=1}^{k+1} (-1)^{j-1}  (2^{2j}-1)  (2^{2k+2-2j}-1) \frac{B_{2j}}{(2j)!}  \frac{B_{2k+2-2j }}{(2k+2-2j)!} \alpha^j \beta^{k+1-j}.   
\end{align}
This identity is nothing but the equation \eqref{Ramana}.  Here we point out that the term corresponding to $j=k+1$ vanishes due to the presence of the factor $(2^{2k+2-2j}-1)$.  Now plugging $\tanh z = 1- \frac{2}{e^{2z} +1}$ and multiply by $(-1)^{k+1}$ in \eqref{alternative form_Ramanujan},  we deduce that
\begin{align}\label{odd zeta formula}
& \frac{1}{\alpha^k }\sum_{n=0}^\infty \left(  \frac{1}{(2n+1)^{2k+1} } - \frac{2 (2n+1)^{-2k-1}}{e^{ 2(2n+1)\alpha} + 1} \right) + \frac{(-1)^{k+1}}{\beta^k }\sum_{n=0}^\infty \left(  \frac{1}{(2n+1)^{2k+1} } - \frac{2 (2n+1)^{-2k-1}}{e^{ 2(2n+1)\beta} + 1} \right)  \nonumber \\ 
&= 2^{2k+1}   \sum_{j=1}^{k} (-1)^{k+j}  (2^{2j}-1)  (2^{2k+2-2j}-1) \frac{B_{2j}}{(2j)!}  \frac{B_{2k+2-2j }}{(2k+2-2j)!} \alpha^j \beta^{k+1-j}.  
\end{align}
Finally,  observe that
\begin{align}\label{Zeta(2k+1)}
\zeta(2k+1)\left( 1- 2^{-2k-1} \right) = \sum_{n=0}^\infty \frac{1}{(2n+1)^{2k+1}}.
\end{align}
Now making use of \eqref{Zeta(2k+1)} and replacing $j$ by $k+1-j$ in \eqref{odd zeta formula},  one can obtain \eqref{eq_analogue_ramanujan}.
\end{proof}

\begin{proof}[Corollary {\rm \ref{tan particular case}}][]
Letting $\alpha=\beta= \pi/2$, replacing $k$ by $2k+1$ in \eqref{Ramana} yields \eqref{Tan particular case}.  To obtain \eqref{eq_analogue_lerch},  we have to use the identity $\tanh z = 1- \frac{2}{e^{2z} +1}$ in \eqref{Tan particular case}.  
\end{proof}

\begin{proof}[Theorem {\rm \ref{thm_gen_rama_analogue} }][]
First,  we point out that the infinite series 
$$
\summ{n=1}{\infty}{a_n \left( \sum_{d|n} \frac{\chi_1(d)}{d^{2k+1}} \right) }e^{-2n x}  
$$
is convergent for any positive real number $x$.  Making use of the inverse Mellin transform for $\Gamma(s)$ and employing Lemma \ref{lem_zeta_l_sum},  for $\Re(s)=c > \max\{1, -2k\}$,  we can readily show that
\begin{align}\label{series_integral}
\summ{n=1}{\infty}{a_n \left( \sum_{d|n} \frac{\chi_1(d)}{d^{2k+1}} \right) }e^{-2n x}  =  \frac{1}{2\pi i} \I{c} \Gamma(s)\zeta(s)(1-p^{1-s})L(s+2k+1,\chi_1)(2x)^{-s} \mathrm{d}s,
\end{align}
where $\chi_1$ is the principal Dirichlet character modulo $p$.  Note that $L(s+2k+1,\chi_1) = \rz{s+2k+1}(1-p^{-s-2k-1})$.  Thus,  we must try to evaluate the following line integration
\begin{align}\label{right_vertical_integral}
I_{p,  k}(x):= \frac{1}{2\pi i} \I{c} \Gamma(s)\zeta(s)  \zeta(s+2k+1) (1-p^{1-s}) (1-p^{-s-2k-1}) (2x)^{-s} \mathrm{d}s. 
\end{align}
Let us define the integrand function $$f(s):= \Gamma(s)\zeta(s)  \zeta(s+2k+1) (1-p^{1-s}) (1-p^{-s-2k-1}) (2x)^{-s}.  $$
Now we shall try to investigate the poles of the integrand function.  We know that $\Gamma(s)$ has simple poles at negative integers including zero and $\zeta(s)$ has a simple pole at $s=1$.  Again,  $\zeta(s)$ has trivial zeros at negative even integers.  Note that $(1-p^{1-s}) (1-p^{-s-2k-1})$ has trivial zeros at $s=1$ and $s=-(2k+1)$.  \\
{\bf Case I}: For $k>0$,  one can check that the only simple poles of the integrand function $f(s)$  are at $0,  -1,  -3,  -5,  \cdots,  -2k+1$,  and $-2k$. \\
{\bf Case II}: When $k$ is a negative integer,  $\zeta(s+2k+1)$ has trivial zeros at all negative odd integers.  Therefore,  in that case,  the integrand function $f(s)$ has poles only at $0$ and $-2k$.  

Now we construct a suitable rectangular contour $\mathfrak{C}$ consisting of the vertices $c-i T,  c+i T, d+i T,  d-iT$ with counter-clockwise direction.  Here $T$ is some large positive quantity with $c > \max\{ 1, -2k\}$  and $d$ is cleverly chosen to be less than $\min\{ 0,  -2k-1 \}$ so that all the poles of the integrand function $f(s)$ lie inside the contour $\mathfrak{C}$.  
Now appealing to Cacuchy's residue theorem,  we see that 
\begin{align}\label{CRT}
   \frac{1}{2\pi i}\bracket{ \int_{c-iT}^{c+iT}+ \int_{c+iT}^{d+iT} + \int_{d+iT}^{d-iT} + \int_{d-iT}^{c-iT}  }f(s) \mathrm{d}s = R_{0} + R_{-2k} + \mathfrak{R}(x),
\end{align}
where 
\begin{align}
\mathfrak{R}(x)= \begin{cases}\displaystyle
\sum_{j=1}^{k} R_{-(2j-1)}, & {\rm if} \,\,  k>0,  \\
0, & {\rm if} \,\,  k<0.  
\end{cases}
\end{align}
and $R_\rho$ denotes the residual term corresponding to the pole at $s=\rho$.  Now we evaluate the residual terms,  i.e., 
\begin{align}
 R_0=   \lim_{s \to 0}s f(s) &= \lim_{s\to 0} s \Gamma(s)\zeta(s)\zeta(s+2k+1)(1-p^{1-s})(1-p^{-s-2k-1})(2x)^{-s} \nonumber \\
    &= \zeta(0)\zeta(2k+1)(1-p)(1-p^{-2k-1})
    = \frac{(p-1)  \ls{2k+1}}{2},  \label{R_0}
\end{align}
\begin{align}
    R_{-2k} &= \lim_{s\to -2k}(s+2k)\Gamma(s)\zeta(s)\zeta(s+2k+1)(1-p^{1-s})(1-p^{-s-2k-1})(2x)^{-s} \nonumber \\
    &=(1-p^{1+2k})(1-p^{-1})\frac{\zeta^{'}(-2k) (2x)^{2k} }{(2k)!} \nonumber  \\
    &=(-1)^{k+1} (p-1)\bracket{\frac{px}{\pi}}^{2k} \frac{\ls{2k+1}}{2},  \label{R_{-2k}}
\end{align}
to obtain the final step we have used the identities $\zeta^{'}(-2k)= (-1)^k \frac{(2k)!}{2 (2\pi)^{2k}} \zeta(2k+1)$ and $L(2k+1, \chi_1) = \zeta(2k+1) (1-p^{-2k-1})$.  Again,  we calculate
{\allowdisplaybreaks
\begin{align}
\sum_{j=1}^{k} R_{-(2j-1)} 
&=\sum_{j=1}^{k}  \lim_{s\to - (2j-1)}(s + 2j-1)\Gamma(s)\zeta(s)\zeta(s+2k+1)(1-p^{1-s})(1-p^{-s-2k-1})(2x)^{-s}  \nonumber \\
     &= \sum_{j=1}^{k} \frac{(-1)^{2j-1}}{(2j- 1)!}\zeta(-2j+1)\zeta(2k-2j +2)(1-p^{2j})(1-p^{2j-2k-2})(2x)^{2j-1} \nonumber \\
     &=\sum_{j=1}^{k}  \frac{B_{2j}}{(2j)!}(1-p^{2j})(1-p^{2j-2k-2})\zeta(2k-2j+2)(2x)^{2j-1} \nonumber \\
     &=\sum_{j=1}^{k}  \frac{(-1)^{k+j}}{2} (1-p^{2j})(1-p^{2j-2k-2}) \frac{B_{2j}}{(2j)!}\frac{B_{2k-2j+2}  }{(2k-2j+2)!}(2\pi)^{2k-2j+2} (2x)^{2j-1}  \nonumber  \\
     & = \sum_{j=1}^{k} \frac{(-1)^{j} }{2} \left(p^{2j}-1 \right)  \left(p^{2k+2-2j}-1\right) \frac{B_{2j}}{(2j)!}\frac{B_{2k+2-2j}}{(2k+2-2j)!} (2x)^{2k+1} \left( \frac{\pi}{p x} \right)^{2j}. \label{R_(-2j+1)}
\end{align}}
To obtain the final expression,  in the penultimate step we have used Euler's formula \eqref{euler's identity} and in the ultimate step we replaced the variable $j$ by $k+1-j$.  
Now with the help of Stirling's formula \eqref{Stirling's formula} for $\Gamma(s)$ and  estimate for $\zeta(s)$,  one can show that the contribution of the horizontal integrals vanish as $T \rightarrow \infty$.  Thus,  letting $T \to \infty$ in \eqref{CRT},  we arrive at
\begin{align}\label{application of CRT}
\frac{1}{2\pi i} \int_{c-i\infty}^{c+i\infty} f(s) \mathrm{d}s = \frac{1}{2\pi i} \int_{d-i\infty}^{d+i\infty} f(s) \mathrm{d}s + R_{0} + R_{-2k} + \mathfrak{R}(x). 
\end{align}
At this moment,  one of our main goals is to simplify the following left vertical integral 
\begin{align}
J_{p,k}(x) :=  \frac{1}{2\pi i} \int_{d-i\infty}^{d+i\infty} \Gamma(s)\zeta(s)  \zeta(s+2k+1) (1-p^{1-s}) (1-p^{-s-2k-1}) (2x)^{-s} \mathrm{d}s,
\end{align}
where $d < \min\{0, -2k\}$.  In order to write this integral as an infinite series we will make a sutiable change of variable,  as discussed below.  Replace $s$ by $-s-2k$ to obtain
\begin{align}\label{left_vert}
J_{p,k}(x) & =  \frac{1}{2\pi i} \int_{(-d-2k)}\Gamma(-s-2k)\zeta(-s-2k)\zeta(1-s)(1-p^{1+s+2k})(1-p^{s-1})(2x)^{s+2k} \mathrm{d}s.
\end{align}
Here we note that the new line of integration $\Re(s)= -d-2k >\max\{1, -2k \}$ as we have considered $d < \min\{0, -2k\}$.  Now employing Lemma \ref{Use of functional equation} in \eqref{left_vert}  the left vertical integral $J_{p,k}(x) $ takes the shape as
\begin{align*}
 \bracket{\frac{p x}{\pi}}^{2k}  \frac{(-1)^k}{2\pi i} \int_{(-d-2k)} \Gamma(s)\zeta(s)\zeta(s+2k+1)(1-p^{1-s})(1-p^{-s-2k-1})\bracket{\frac{ 2 \pi^2}{p^2 x}}^{-s} \mathrm{d}s. 
\end{align*}
In view of \eqref{series_integral} and \eqref{right_vertical_integral},  we can clearly see that the above integral is equals to
\begin{align}
J_{p,k}(x) 
& = (-1)^k  \bracket{\frac{p x}{\pi}}^{2k} I_{p,k}\left( \frac{ \pi^2}{p^2 x}  \right) \nonumber \\
& = (-1)^k  \bracket{\frac{p x}{\pi}}^{2k} \summ{n=1}{\infty}{a_n \left( \sum_{d|n} \frac{\chi_1(d)}{d^{2k+1}} \right) }e^{- \frac{2n \pi^2}{p^2 x} }.  \label{Final_left_vertical}
\end{align}
Eventually,  combining \eqref{series_integral},  \eqref{application of CRT} and \eqref{Final_left_vertical} and collecting all the residual terms \eqref{R_0},  \eqref{R_{-2k}},   \eqref{R_(-2j+1)},  and substituting $ x = \alpha$ and $\alpha \beta = \pi^2/ p^2$ and simplifying further,  one can complete the proof of Theorem \ref{thm_gen_rama_analogue}. 
\end{proof}

\begin{proof}[Corollary {\rm \ref{cor_gen_lerch_analogue}}][]
We substitute $\al = \be = \frac{\pi}{p}$ and replace $k$ by $ 2k+1 $ as a positive odd integer in  Theorem \ref{thm_gen_rama_analogue} to complete the proof of this corollary.
\end{proof}

\begin{proof}[Corollary {\rm \ref{for k negative}}][]
Considering $k$ as a negative integer i.e.,  replacing $k$ by $-(k+1)$  in Theorem \ref{thm_gen_rama_analogue},  for $k\geq 0$,  we obtain
\begin{align}
   &(4 \al)^{k+1}\bracket{  \frac{p-1}{2}L(-2k-1,\chi_1)- \summ{n=1}{\infty} {a_n \bigg( \sum_{d|n} {\chi_1(d) d^{2k+1}} \bigg) }e^{-2n \al}  } \nonumber \\
    &  = (-4 \be)^{k+1}\bracket{  \frac{p-1}{2}L(-2k-1,\chi_1)- \summ{n=1}{\infty} {a_n \bigg( \sum_{d|n} {\chi_1(d) d^{2k+1}} \bigg) }e^{-2n \beta}  }. \nonumber 
\end{align}
Now using the fact that $L(-2k-1, \chi_1)= (p^{2k+1}-1) \frac{B_{2k+2}}{2k+2}$ and simplifying,  one can finish the proof.  
\end{proof}

\begin{proof}[Theorem {\rm \ref{thm_k=0}}][]
The proof of this identity goes along the same line as in Theorem \ref{thm_gen_rama_analogue},  so we will mention those important steps where it differs from the previous one.  In view of \eqref{series_integral} and \eqref{right_vertical_integral},  with $k=0$,  we can show that 
\begin{align}\label{right_vertical_k_0}
\summ{n=1}{\infty}{a_n \left( \sum_{d|n} \frac{\chi_1(d)}{d} \right) }e^{-2n x}  =  \frac{1}{2\pi i} \I{c} \Gamma(s)\zeta(s)(1-p^{1-s})L(s+1,\chi_1)(2x)^{-s} \mathrm{d}s,
\end{align}
where $c>1$ and $L(s+1,  \chi_1)= )\zeta(s+1)(1-p^{-s-1})$.  Thus,  in this case,  our integrand function is 
$$ f(s) = \Gamma(s)\zeta(s)\zeta(s+1)(1-p^{1-s})(1-p^{-s-1})(2x)^{-s}. $$ 
One can notice that the poles of $\Gamma(s)$ at negative even integers are getting neutralized by the trivial zeros of $\zeta(s)$,  and the poles of $\Gamma(s)$ at negative odd integers will get cancelled by the trivial zeros of $\zeta(s+1)(1-p^{-s-1})$.  Hence,  the only pole of the integrand function $f(s)$ is at $s=0$ of order $2$.  Now we would like to shift the line of integration to the left,  on line $\Re(s)=d$ with $d<0$.  Invoking Cauchy's residue theorem,  one can show that 
\begin{align}\label{CRT_k_0}
\frac{1}{2\pi i} \I{c} f(s)  \mathrm{d}s = \frac{1}{2\pi i} \I{d} f(s)  \mathrm{d}s + R_0 + R_1,
\end{align}
where $R_0$ and $R_1$ denote the residual terms corresponding to the poles at $s=0$ and $s=1$ respectively.  Now we jot down the Laurent series expansions of each term present in $f(s)$ around $s=0$,  namely,
 \begin{align*}
    \Gamma(s) &= \frac{1}{s} - \gamma + \frac{1}{2}\bracket{\gamma^2 + \frac{\pi^2}{6}} s + \cdots,\\
    \zeta(s)  &= -\frac{1}{2}+\zeta^{'}(0)s+\cdots,\\
    \zeta(s+1) &= \frac{1}{s} + \gamma -\gamma_1 s+ \cdots,\\
    (1-p^{1-s})&= 1-p + (p \log p)s - \frac{1}{2} p(\log p)^2  s^2 + \cdots,\\
    (1-p^{-s-1})&=\bracket{1-\frac{1}{p}}+ \bracket{\frac{\log p}{p}}s+ \cdots,\\
    (2x)^{-s}&= 1+ \bracket{\log\frac{1}{2x}}s + \frac{1}{2}\bracket{ \log \frac{1}{2x} }^2 s^2 + \cdots,
\end{align*}
where $\gamma$ is the Euler-Mascheroni constant, and $\gamma_n$ are the Stieltjes constants. Utilizing the above Laurent series expansions,  one can see that the following Laurent series expansion of $f(s)$ around $s=0$ holds,
 \begin{equation} \label{eq_laurent_for_f(s)}
     f(s) = \frac{(p-1)^2}{2p}\frac{1}{s^2}+ \frac{(p-1)^2}{2p}\log\bracket{\frac{\pi}{px}}\frac{1}{s}+ \cdots. 
 \end{equation}
This indicates that 

 \begin{equation}\label{residue_R_0}
        R_0= \lim_{s \rightarrow 0} \frac{{\rm d} }{ {\rm d}s}s^2 f(s) = \frac{(p-1)^2  }{2p}\log\bracket{\frac{\pi}{p x}} .
    \end{equation}
The simplification of the left vertical integral goes along the same vein as in Theorem \ref{thm_gen_rama_analogue},  so we won't repeat the calculation here.  Mainly,  in view of \eqref{Final_left_vertical},  we reach at 
\begin{align}\label{Final_left_k_0}
J_{p,0}(x) =   \frac{1}{2\pi i} \I{d} f(s)  \mathrm{d}s=  \summ{n=1}{\infty}{a_n \left( \sum_{d|n} \frac{\chi_1(d)}{d} \right) }e^{- \frac{2n \pi^2}{p^2 x} }.
\end{align} 
Ultimately,  combining \eqref{right_vertical_k_0},  \eqref{CRT_k_0},  \eqref{Final_left_k_0},  and together with the residual term $R_0$ in \eqref{residue_R_0},  and substituting $x=\alpha$ and $\alpha \beta =\pi^2/p^2$,  we settle the proof of \eqref{identity_k_0}.  
\end{proof}

\begin{proof}[Theorem {\rm \ref{thm_gen_grosswald_analogue}}][]
First,  we note that Theorem \ref{thm_gen_rama_analogue} can be analytically continued for $\Re(\alpha) >0,  \Re(\beta)>0$.   Multiplying by $(4\alpha)^k$ and substituting $\alpha = - i \pi z$ and $\beta= \frac{i \pi}{p^2 z}$ on both sides of Theorem \ref{thm_gen_rama_analogue} yields 
\begin{align*}
\frac{p-1}{2}L(2k+1,\chi_1) & \left\{ 1 + (-1)^{k+1} \left(\frac{\alpha}{\beta} \right)^k  \right\} = \mathfrak{F}_{2k+1, \chi_1}(z) + (-1)^{k+1} \left(\frac{\alpha}{\beta} \right)^k \mathfrak{F}_{2k+1,  \chi_1}\left(-\frac{1}{p^2 z} \right)  \nonumber \\
& + 2^{2k} \summ{j=1}{k}  {(-1)^{j-1}(p^{2j}-1)(p^{2k+2-2j}-1)\frac{B_{2j}}{(2j)!}\frac{B_{2k+2-2j}}{(2k+2-2j)!}\al^{2k+1-j}\be^j },
\end{align*}
where $ \mathfrak{F}_{k,\chi_1}(z)= \sum_{n=1}^\infty a_n \sigma_{-k,  \chi_1}(n) \exp( 2\pi i n z).$
Simplifying further, it turns out that
\begin{align*}
   \frac{p-1}{2}L(2k+1,\chi_1) & \left\{ 1 - \left( z p\right)^{2k}  \right\} = \mathfrak{F}_{2k+1, \chi_1}(z) - \left( z p\right)^{2k} \mathfrak{F}_{2k+1,  \chi_1}\left(-\frac{1}{p^2 z} \right)  \nonumber \\
   + & \frac{(2\pi i)^{2k+1}}{2z} \summ{j=1}{k}  {(1-p^{-2j})(p^{2k+2-2j}-1)\frac{B_{2j}}{(2j)!}\frac{B_{2k+2-2j}}{(2k+2-2j)!} }z^{2k+2-2j}. 
\end{align*}
Upon simplification,  one can complete the proof. 
\end{proof}

\begin{proof}[Proposition {\rm \ref{thm_gmr_one_alg}}][]
This result can be proved by the contradiction method.  Let us suppose that there are at least two distinct complex algebraic numbers $z_1$ and $z_2$,  satisfying $(zp)^{2k} \neq 1$,  such that $\mathfrak{G}_{2k+1,  \chi_1}(z_1)$ and $\mathfrak{G}_{2k+1,  \chi_1}(z_1)$ are both algebraic.  Now, Theorem \ref{thm_gen_grosswald_analogue} indicates that 
\begin{align}
\mathfrak{G}_{2k+1,  \chi_1}(z_1) & =  \frac{p-1}{2}L(2k+1,\chi_1)  +  \frac{(2 \pi i)^{2k+1} }{2 z_1 \, p^{2k+2}} \frac{ \mathfrak{R}_{2k+1, p}(z_1)}{ \left( (pz_1)^{2k}-1  \right)  },  \\
\mathfrak{G}_{2k+1,  \chi_1}(z_2) & =  \frac{p-1}{2}L(2k+1,\chi_1)  +  \frac{(2 \pi i)^{2k+1} }{2 z_2 \, p^{2k+2}} \frac{\mathfrak{R}_{2k+1, p}(z_2)}{ \left( (pz_2)^{2k}-1  \right) } . 
\end{align}
Subtracting the second equation from the first,  we get
\begin{align}
\mathfrak{G}_{2k+1,  \chi_1}(z_1)-\mathfrak{G}_{2k+1,  \chi_1}(z_2) = \frac{(2 \pi i)^{2k+1} }{2 p^{2k+2}} \left( \frac{\mathfrak{R}_{2k+1, p}(z_1)}{z_1\left( (pz_1)^{2k}-1  \right)} -  \frac{\mathfrak{R}_{2k+1, p}(z_2)}{z_2\left( (pz_2)^{2k}-1  \right)} \right). 
\end{align}
Note that the above right side expression is transcendental, since $\pi$ is transcendental and $\mathfrak{R}_{2k+1}(z)$ is algebraic for $z$ being an algebraic number,  whereas the left side expression is an algebraic number,  which arises a contradiction.  Thus, no such distinct $z_1, z_2$ exists.  This finishes the proof.  

\end{proof}

\begin{proof}[Theorem {\rm \ref{thm_gmr_exception}}][]
Let $z \in \mathbb{H}\cup \overline{\mathbb{Q}}$ with $(p z)^{2k} \neq 1$.  Proposition \ref{thm_gmr_one_alg} implies that the following set
$$ \{  \mathfrak{G}_{2k+1,\chi_1}(z) | \hspace{.1cm} z \in \mathbb{H} \cap \overline{ \mathbb{Q}}, \hspace{.2cm} (pz)^{2k} \neq 1 \} $$
contains at most one algebraic number.  Let us call that algebraic number as $A$.  Now we define a set
\begin{align}
S:= \left\{ z \in \mathbb{H} \cap \overline{ \mathbb{Q}}, \hspace{.2cm} (pz)^{2k} \neq 1 \big| \mathfrak{G}_{2k+1,\chi_1}(z) =A  \right\}.
\end{align}
Theorem \ref{thm_gen_grosswald_analogue} indicates that the elements of the set $S$ satisfy the following polynomial: 
\begin{align}
A \left\{ ( zp)^{2k}-1 \right\} = \frac{p-1}{2}L(2k+1,\chi_1)  \left\{ \left( z p\right)^{2k} -1  \right\} 
   + & \frac{(2\pi i)^{2k+1}}{ 2z \, p^{2k+2}} \mathfrak{R}_{2k+1, p}(z),
\end{align}
which is a polynomial of degree $2k$.  Therefore,  we have at most $2k$ algebraic numbers $z \in \mathbb{H}$ that are satisfying $(z p)^{2k}\neq 1$ for which $\mathfrak{G}_{2k+1,\chi_1}(z)$ is algebraic,  that is,  the quantity $ \left( z p\right)^{2k}   \mathfrak{F}_{2k+1,  \chi_1}\left(-\frac{1}{p^2 z} \right)  - \mathfrak{F}_{2k+1, \chi_1}(z)$ is algebraic. 

On the other hand,  we consider $z \in \mathbb{H} \cup \mathbb{\overline{\mathbb{Q}}}$ that are satisfying $(zp)^{2k}=1$.  
Now Theorem \ref{thm_gen_grosswald_analogue} gives us that
\begin{align*}
\left( z p\right)^{2k} &  \mathfrak{F}_{2k+1,  \chi_1}\left(-\frac{1}{p^2 z} \right)  - \mathfrak{F}_{2k+1, \chi_1}(z)= \frac{(2\pi i)^{2k+1}}{ 2z \, p^{2k+2}} \mathfrak{R}_{2k+1, p}(z). 
\end{align*}
Note that  the left side expression will be an algebraic number  only when $\mathfrak{R}_{2k+1, p}(z)=0$, otherwise it will arise a contradiction to the fact that $\pi$ is a transcendental number.  Here we employ Conjecture \ref{conjecture2},  that says,  if $2|k$,  the polynomials $\mathfrak{R}_{2k+1, p}(z)$ and $(pz)^{2k}-1$ have only one common root,  namely, $i/p$,  lying in the upper half plane.  Again,  if $2 \nmid k$,  they do not have any common roots.  This gives the clarification of defining $\delta_k$.  Therefore,  combining both of the cases,  we can clearly see that the expression 
$ \left( z p\right)^{2k}   \mathfrak{F}_{2k+1,  \chi_1}\left(-\frac{1}{p^2 z} \right)  - \mathfrak{F}_{2k+1, \chi_1}(z)$ is transcendental,   for any $z \in \mathbb{H}\cup
 \mathbb{\overline{\mathbb{Q}}}$ with at most $2k+\delta_k$ exceptions.  This ends the proof of this theorem.  

\end{proof}

\section{concluding remarks}

Inspired from Ramanujan's idea of proving his famous formula for $\zeta(2k+1)$,  in the current paper,  we have established a new Ramanujan-type formula for $L(2k+1,  \chi_1)$,  namely,  Theorem \ref{thm_gen_rama_analogue}.  And then motivated from Grosswald's generalization of Ramanujan's formula,  we have extended Theorem \ref{thm_gen_rama_analogue} in the upper half-plane and obtained Theorem \ref{thm_gen_grosswald_analogue},  which influenced us to define a new Ramanujan-type polynomial \eqref{new Ramanujan-type polynomila} as 
\begin{align}
\mathfrak{R}_{2k+1,p}(z) =  \summ{j=1}{k}  {(p^{2j}-1)(p^{2k+2-2j}-1)\frac{B_{2j}}{(2j)!}\frac{B_{2k+2-2j}}{(2k+2-2j)!} } (p z)^{2k+2-2j}. 
\end{align}
Replacing $z$ by  $z/p$,  one can check that the above polynomial $\mathfrak{R}_{2k+1,p}(z/p) $ is a reciprocal polynomial.  It is clear that Conjecture \ref{conjecture1} is equivalent of  proving that all the non-real roots $\mathfrak{R}_{2k+1,p}(z/p)$ will lie on $|z|=1$.   To know more about the location of zeros of self-reciprocal polynomials,  readers are encouraged to see  \cite{Chen95,  Lakatos02, Lakatos, Schinzel}.


\begin{table}
\caption{Numerical evidence for Conjecture \ref{conjecture1} and \ref{conjecture2}. In the below table, we have evaluated zeros of a few Ramanujan-type polynomial  $\mathfrak{R}_{2k+1,p}(z/p)$. 
  }
\label{Numerical_Zeros}
\renewcommand{\arraystretch}{1}
{\small
\begin{tabular}{|l|l|l|l|l|l|l}
\hline
  $p$ &  $k$ & $\mathfrak{R}_{2k+1,p}(z/p)$ & Real roots  & Non-real roots  \\
  \hline 
  $2$ & $1$  & $\frac{z^2}{16}$  & $ 0,0$  & - \\
  \hline
 $2$ &  $2$  & $ -\frac{z^2}{192} - \frac{z^4}{192}$   & $ 0,0$ & $ \pm i$  \\      
\hline
 $3$ & $4$ & $- \frac{41 z^2}{11340} - \frac{13 z^4}{4860} - \frac{13 z^6}{4860} - \frac{41 z^8}{11340} $  & $0,0$ & $\pm i,  \pm \sqrt{\frac{16}{123} \pm \frac{i\sqrt{14873}}{123}} $ \\
\hline
$5$ &  $3$ & $\frac{31 z^2}{30} + \frac{169 z^4}{225} + \frac{31 z^6}{30}$  & $0,0$ & $\pm \sqrt{-\frac{169}{465} \pm \frac{4 i \sqrt{11729}}{465}}$ \\
\hline
 $5$ & $4$ & $- \frac{4069 z^2}{6300} - \frac{403 z^4}{900} - \frac{403 z^6}{900} - \frac{4069 z^8}{6300} $  & $0,0$ & $ \pm i,  \pm \sqrt{\frac{48}{313}  \pm \frac{19\, i \sqrt{265}}{313}}$ \\
\hline
 $7$ & $3$ & $ \frac{1634 z^2}{105} + \frac{100 z^4}{9} + \frac{1634 z^6}{105} $  & $ 0,0 $  & $ \pm \sqrt{-\frac{875}{2451} \pm \frac{4 i \sqrt{327611}}{2451}}$ \\
\hline
$7$ & $4$ & $ -\frac{1201 z^2}{63} - \frac{817 z^4}{63} - \frac{817 z^6}{63} - \frac{1201 z^8}{63} $ & $ 0,0 $  & $ \pm i, \pm \sqrt{\frac{192}{1201} \pm \frac{i \sqrt{1405537}}{1201}}$\\        
\hline
 
\end{tabular}}

\end{table}

One can easily verify that all complex zeros of the polynomials, mentioned in the Table \ref{Numerical_Zeros}, are lying on $|z|=1$.  While collecting these numerical data we have observed that  Conjecture \ref{conjecture1} might be true,  even if we replace the prime number $p$ by any natural number $n \geq 2$.  

 The  Table  \ref{Numerical_Zeros} also indicates that, $\pm i$ are the only zeros of $\mathfrak{R}_{2k+1,p}(z/p)$ that are $2k$th roots of unity  when $k$ is divisible by $2$.  This  suggests that Conjecture \ref{conjecture2} is also valid for the polynomials mentioned in Table  \ref{Numerical_Zeros}.

\begin{table}[h]
\caption{Verification of Theorem \ref{thm_gen_rama_analogue}:  The left-hand side sum over $n$ considered only first $1000$ terms with specific values of $p$, $k$ and $\al$, $\be$.  We have used Mathematica software to evaluate this numerical data.}
\label{Table_Ramanujan cusp form}
\renewcommand{\arraystretch}{1}
{\small
\begin{tabular}{|l|l|l|l|l|l|}
\hline
  $p$ & $k$ &$\al$ &$\be$ & Left-hand side  & Right-hand side  \\
  \hline
$2$   &$1$ &$\frac{\pi}{2}$ &$\frac{\pi}{2}$  &$0.154212568767021... $  &$0.154212568767021...$\\
\hline
$2$   &$2$ &$1$ &$\frac{\pi^2}{4}$  &$0.018857641090379... $  &$0.018857641090379...$\\
\hline
$2$   &$2$ &$\pi$ &$\frac{\pi}{4}$  &$ -0.030279567070605... $  &$-0.030279567070605...$\\
\hline
$2$   &$3$ &$\frac{\pi}{2}$ &$\frac{\pi}{2}$  &$ 0.003699346990223... $  &$0.003699346990223...$\\
\hline
$2$   &$3$ &$\pi$ &$\frac{\pi}{4}$  &$ 0.010833801899940... $  &$0.010833801899940...$\\
  \hline
$3$   &$2$ &$\pi$ &$\frac{\pi}{9}$  &$ -0.226840615859532... $  &$-0.226840615859532...$\\
\hline
$3$   &$3$ &$\pi$ &$\frac{\pi}{9}$  &$ 0.161004035376020... $  &$0.161004035376020...$\\
\hline
$5$   &$1$ &$\frac{\pi}{5}$ &$\frac{\pi}{5}$  &$1.579136704174297... $  &$1.579136704174297...$\\
\hline
$5$   &$3$ &$\pi$ &$\frac{\pi}{25}$  &$ 3.915620336334279... $  &$3.915620336334286...$\\
\hline
\end{tabular}}

\end{table}


{\bf Acknowledgement:} The authors would like to thank Prof.  Bruce Berndt and Prof.  Atul Dixit for giving valuable suggestions.  We are also thankful to the Computational Number Theory (CNT) Lab, IIT Indore for providing conductive research environment. 
 The third author wants to thank SERB for the Start-Up Research Grant SRG/2020/000144.

\end{document}